\documentclass[reqno]{amsart}
\usepackage{setspace}
\usepackage{amsmath}
\usepackage{amssymb}
\usepackage{mathtools}
\makeatletter
\def\bc{\begin{center}}
\def\ec{\end{center}}

\def\s2c{\vskip 2cm}
\def\bt{\begin{Theorem}}
\def\et{\end{Theorem}}
\def\bd{\begin{Definition}}
\def\ed{\end{Definition}}
\def\bl{\begin{Lemma}}
\def\el{\end{Lemma}}
\def\bcor{\begin{Corollary}}
\def\ecor{\end{Corollary}}
\def\bpr{\begin{Proposition}}
\def\epr{\end{Proposition}}

\newtheorem{Lemma}{Lemma}[section]
\newtheorem{Theorem}[Lemma]{Theorem}
\newtheorem{Proposition}[Lemma]{Propostion}
\newtheorem{Definition}[Lemma]{Definition}
\newtheorem{Corollary}[Lemma]{Corollary}
\newtheorem{theorem}{Theorem}[section]
\newtheorem{lemma}[theorem]{Lemma}

\newtheorem{definition}[theorem]{Definition}

\def\elsartstyle{%
    \def\normalsize{\@setfontsize\normalsize\@xiipt{14.5}}
    \def\small{\@setfontsize\small\@xipt{13.6}}
    \let\footnotesize=\small
    \def\large{\@setfontsize\large\@xivpt{18}}
    \def\Large{\@setfontsize\Large\@xviipt{22}}
    \skip\@mpfootins = 18\p@ \@plus 2\p@
    \normalsize
} \@ifundefined{square}{}{} \makeatother

\author{ Subha Pal and Rajib Haloi$^*$}
\date{}

\thanks{$^*$Corresponding author's e-mail:rajib.haloi@gmail.com, }
\begin{document}

\title{Existence and uniqueness of solutions to the damped Navier-Stokes equations with Navier  boundary conditions
for three dimensional incompressible fluid}
\maketitle \vskip .5cm \noindent {\bf Abstract:}
 {In this article, we study the   solutions of the  damped Navier--Stokes equation with Navier boundary condition 
 in a bounded domain $\Omega$ in $\mathbb{R}^3$ with smooth boundary.
 The existence of the solutions is global with the damped term $\vartheta |u|^{\beta-1}u, \vartheta >0.$ 
 The regularity and uniqueness of  solutions with Navier boundary condition is also studied. This extends the  existing results in literature.}

\vskip .3cm \noindent {\bf AMS Classification (2010):}{ 76D05,35A01, 35Q30,76D03.}

\section{Introduction}

In this article, we consider the following system of Navier-Stokes (N-S) equations in a simply connected bounded  domain
$\Omega$ of $\mathbb{R}^3$,  with sufficiently smooth boundary,
\begin{align}
 \partial_{t}u + u \cdot \nabla u + \vartheta |u|^{\beta - 1}u 
+\frac{1}{\rho}\nabla p -
 \mu \Delta u &= f \qquad  \text{in} ~ \ Q_T ,\label{be}\\
{\rm div} \ u &= 0 \qquad  \text{in} ~ \ Q_T ,\label{inc}\\
 u&=u_{0} \quad \hspace{.2cm} \text{in} ~ \ \Omega \times \{ 0 \},\label{inc0}\\
u \cdot \nu &=0 \qquad  \text{on} ~ \ \partial \Omega \times 
(0,T),\\
2D(u)\nu \cdot \tau + \alpha u \cdot \tau &=0 \qquad  \text{on} ~ \ \partial \Omega \times 
(0,T),\label{nosl}
\end{align}
where  $Q_T = \Omega \times (0,T), T> 0$. Here the unknown function $p = p(x,t)$ is the pressure and 
$u = u(x,t) = (u_1(x,t),u_2(x,t),u_3(x,t))$ is the velocity of the flow. $\beta \geq 1$ and $\vartheta > 0$ are two constants
in the  damping term $\vartheta |u|^{\beta - 1}u $.
 $\rho$ is the density, $f$ is the external force and 
$\mu > 0$ is the kinematic viscosity, $u_0$ is the initial velocity of the fluid. 
The rate of strain tensor $D(u)$ is   $D(u)=\frac{1}{2}\big[\nabla u+(\nabla u)^T\big]$, $\alpha (x) >0$
is a continuously differentiable  function  on the boundary $\partial \Omega$, $\nu$ and $\tau$
are unit exterior normal and the unit tangent vector to the boundary. The balance of momentum of the motion of the fluid 
is given by (\ref{be}) and the condition of incompressibility is expressed in  (\ref{inc}). The Navier slip boundary condition is 
given by  (\ref{nosl}). 

The  Navier-Stokes (N-S) equations models of important phenomena in fluid mechanics and describe the motion of fluid like water, air, oil in more general condition. So, it is important to the study of the various form of the solution of the N-S equations with different boundary conditions and are of practical importance. The analytical solutions  of non-linear   N-S equations are complicated to obtain except in some few simple situations. Even solutions for  simple situations  needs very high tools of mathematics.

The no-slip boundary condition ($u=0$ on the boundary) shows us the breakdown of the traditional macroscopic ideas at small scales
between the fluid and solid interface \cite{mat06}. To hold no-slip condition at the fluid-solid interface, it is necessary that the fluid is a continuum and the flow is in thermodynamic equilibrium. The no-slip condition means the fluid velocity relative to the solid boundary is zero. 
 This happens only if the flow is in thermodynamic equilibrium. This needs an infinitely high frequency of collisions between the fluid and the solid surface.  So, the tangential velocity slip must be allowed \cite{MH07} in the proportion
of the tangential component of the stress. Thus the  no-slip boundary condition is replaced by the Navier slip boundary condition
 \cite{mat06, max79, nav27}. We express the Navier slip condition  on $\partial\Omega \times (0,T)$ as 
\begin{eqnarray}
   u \cdot n &= &0,\label{sl1} ~ \\
   2D(u)\nu \cdot \tau + \alpha u \cdot \tau &=& 0.\label{sl2} 
\end{eqnarray}
where $2D(u)\nu \cdot \tau$ denotes the tangential component of 
$2D(u) \cdot n $ and $\alpha$ is 
the coefficient of proportionality.  The authors in  \cite{IS11,zki04,kas13,jm01,vc10,vc12,CA17} has studied the existence of 
solutions to N-S equations with  Navier slip boundary condition.
Clopeau \emph{et al.} \cite{cmr98} studied the regular solutions 
 for  two-dimensional incompressible Navier-Stokes equations with Navier
boundary conditions. Filho \emph{et al.}\cite{flp05} extend the work of 
Clopeau \emph{et al.} \cite{cmr98} to study the vorticities in $L^p$ with $p > 2$. 
Kelliher \cite{kel06} extends the work of  
Clopeau \emph{et al.} \cite{cmr98} and Filho \emph{et al.} \cite{flp05} 
for the bounded domain in $\mathbb R^2$.

 The system (\ref{be})--(\ref{nosl}) with $\vartheta > 0 $ describe the flow 
with the resistance to the motion such as porous media flow and drag or friction effects. From a mathematical viewpoint,
(\ref{be}) can be viewed as a modification of the  Navier-Stokes equation with the regularizing term $\vartheta
|u|^{\beta -1}u $. Thus it is important to study the regularity property and 
uniqueness of the weak solutions \cite{lion69}.  We mention that the system (\ref{be})--(\ref{nosl})  with no slip boundary condition has been studied 
 in \cite{cai08, song12, song15, zhang11,zhou12}.   Cai and Jiu \cite{cai08}  established the existence of a weak solution if $\beta \geq 1$ and 
 global strong solutions if $\beta \geq 7/2$  of the damped Navier-Stokes equation with no slip 
boundary condition. The uniqueness of the  strong solution is proved for $5 \geq \beta \geq 7/2$. For $\beta > 3$, Zhang \emph{et al.} \cite{zhang11},
showed that the  damped damped Navier-Stokes equation with no slip boundary has a global strong. They also obtained the uniqueness of strong the solution when $5 \geq \beta > 3$. Further, they proved that strong solution exists globally for $\beta=3$ and $\vartheta=\mu = 1$,
and the strong solution is unique even among solutions in $\L^{\infty}(0,T;L^2(\Omega)^3)$ for $\beta \geq 1$ in \cite{zhou12}.

In this article, we use the Galerkin approximation method to show the existence of the solutions to the damped N-S equations with Navier slip boundary condition in a simply connected domain of $\mathbb R^3$. It is worth 
noting that so far there are no results in the literature regarding the existence, uniqueness and regularity of the solutions to damped N-S equations with Navier slip boundary condition in $\mathbb R^3$. The article is organised as follows.
The preliminaries and assumptions are provided in Section \ref{preliminaries}. We prove the existence of  solutions in section \ref{weak} and the regularity of solutions in section \ref{exstrong}. We also proved the uniqueness of solution in Section \ref{unique}.

\section{Preliminaries and Assumptions}\label{preliminaries}
In this section we collect the notations, Lemmas and solution spaces. For more details, we refer to Temam \cite{tem79} and Sohr \cite{SH01}.
In the remaining part of the article,  $\Omega$ is  a bounded simply connected domain in $\mathbb{R}^3$ with sufficiently  smooth boundary.
We denote the set of all $C^{\infty}$ real vector function $\phi$ with compact support
in $\Omega$ is denoted by $C_0^{\infty}(\Omega)$. We use the Lebesgue space $L^p(\Omega)~ (1 \leq p \leq \infty)$
and the Sobolev space, 
$H^r(\Omega)=\{\phi \in L^2(\Omega) | \text{weak derivative of}~ \phi ~\text{upto order}~r ~\text{in}~ L^2(\Omega)\}$.
$(\cdot,\cdot)$ denotes $L^2$-inner product and
$$((u,v)) = \sum_{i=1}^n (D_i u, D_i v).$$
The norm of the above two inner-product will be denoted by $\|\cdot\|$ with clear subscripts. We define the following function space like in \cite{ph19}
\begin{center}
$L_0^2(\Omega)=\{z \in L^2(\Omega) : \int_{\Omega} zdx = 0 \},$
$V=\{ v \in H^1(\Omega)^3: {\rm div} \ v = 0 ~~\text{in}~~ \Omega, ~~ v\cdot \nu = 0 ~~\text{on}~~ \partial \Omega \}, $ \\[0.2cm]
$H=\{ v \in L^2(\Omega)^3: {\rm div} \ v = 0 ~~\text{in}~~ \Omega, ~~ v\cdot \nu = 0 ~~\text{on}~~ \partial \Omega\},$  \\[0.2cm]
$\mathcal{W} = \{ v \in V \cap H^2(\Omega)^3 : 2D(u)\nu \cdot \tau + \alpha u \cdot \tau =0 ~~\text{on} ~~
\partial \Omega \}.$
\end{center}
We give $\mathcal{W}$ the $H^2$-norm, $H$ the $L^2$-inner product and norm, which we symbolize by $(\cdot,\cdot)$ 
and $\|\cdot\|_{L^2(\Omega)}$, and $V$ the $H^1$-inner product,
\begin{center}
$((u,v)) = \sum_{i=1}^n (D_i u, D_i v)$
\end{center}
and associate norm. 

\noindent We define the trilinear form $b$ as 
\begin{equation}
b(u,v,w)= \int_{\Omega} (u \cdot \nabla v)\cdot w, ~~~ ~~ \forall u,v,w \in V.
\end{equation}
\begin{equation}
\text{If }u \in V, ~~\text{then}~~ b(u,v,v)=0 ~ , ~ \forall v \in H_0^1(\Omega). \label{p1}
\end{equation}

\noindent For $u,v \in V$, we define $B(u,v)$ by
$$(B(u,v),w)=b(u,v,w),~ \forall w \in V ,$$
and we set $B(u)=B(u,u) \in V' ~,~ \forall u \in V $. So,
\begin{equation}
(Bu,v)= \int_{\Omega} (u \cdot \nabla u) \cdot v \label{p11}.
\end{equation}
We shall use the following Lemma. We omit the proof and refer  to Temam \cite[Lemma III.3.1]{tem79}. 
\begin{lemma} \label{lb1}
Assuming dimension of the space $\leq 4$ and $u \in L^2(0,T;V)$. Then the function $Bu$ defined by
\begin{align*}
(Bu(t),v)=b(u(t),u(t),v), ~~~~ \forall v \in V
\end{align*}
belongs to $L^1(0,T;V')$. Moreover
\begin{equation}
\|Bu\|_{V'} \leq C \|u\|_{H^1}^2,~~~~~ \forall u \in V \label{c2}
\end{equation}
\end{lemma}
\noindent We will use the following lemma from \cite[Lemma 2.1(ii)]{kas13}.
\begin{lemma}\label{pre}
For all $u,v,w \in H_{0}^1(\Omega)$, where $\Omega \in \mathbb{R}^3$ we have 
\begin{align}
|b(u,v,w)| \leq C \|u\|_{L^2(\Omega)}^{\frac{1}{4}} \|u\|_{H^1(\Omega)}^{\frac{3}{4}}\|v\|_{H^1(\Omega)}
\|w\|_{L^2(\Omega)}^{\frac{1}{4}} \|w\|_{H^1(\Omega)}^{\frac{3}{4}}. \label{p2}
\end{align}
\end{lemma}

\noindent Next we define the operator $A$ as follows
\begin{equation}
(Au,v) = 2((u,v)) + \int_{\partial \Omega} \alpha (u \cdot \tau)(v \cdot \tau), ~~ \forall u,v \in V \label{p12}. 
\end{equation}
We have 
\begin{equation}
|(Au,v)| \leq \|u\|_V \|v\|_V +  L_1\|u\|_{L^2(\partial \Omega)}\|v\|_{L^2(\partial \Omega)} \leq L\|u\|_V \|v\|_V, \label{c1}
\end{equation}
where $L_1$ and $L$ are positive constants. So $A : L^2([0,T];V) \rightarrow L^2([0,T];V')$.

The existence of a complete orthogonal basis for $H$ was proved in \cite{ph19}. We recall their result as a  lemma and use it to prove the main theorem. 
\begin{lemma}\label{basis}
There exists a basis $\{ v_n\} \subset H^3(\Omega)^3$, for $V$ which also serves  as an orthonormal basis for
$H$, that satisfies 
\begin{align}
2D(u)\nu \cdot \tau + \alpha u \cdot \tau =0 ~~\text{on} ~~ \partial \Omega.\label{p21}
\end{align} 
\end{lemma}

\section{Existence of solutions} \label{weak}
 In this sections, we prove the existence of solutions to the damped Navier-Stokes system (\ref{be})--(\ref{nosl}) with $\rho =1$. The idea of the proof is similar to that of \cite{tem79} for
the classical Navier-Stokes equations. The definition of solutions is given as usual way.

\begin{definition}
The function $u(x,t)$ is called a solutions of the damped Navier-Stokes system (\ref{be})--(\ref{nosl}), if for any $T > 0$ the 
following conditions are satisfied:\\[0.3cm]
(1) $u \in L^2(0,T;V) \cap L^{\infty}(0,T;H) \cap L^{\beta + 1}(0,T;L^{\beta + 1}(\Omega)),$ \\
(2) for any $v \in V$, we have
   \begin{align} \label{wsol}
   (u',v) &+ 2\mu ((u,v)) + b(u,u,v) + (\vartheta |u|^{\beta -1}u,v) \nonumber  \\
          &+ \mu \int_{\partial \Omega} \alpha (u \cdot \tau) (v \cdot \tau) dS = (f,v). 
   \end{align}
  
\end{definition} 

We state the following  compactness result. For the proof, we refer to Cai and Jiu\cite[Lemma 2.1]{cai08}.

\begin{lemma}\label{lem1}
Let $X_0, X$ be Hilbert space satisfying a compact imbedding $X_0 \xhookrightarrow{} X $. Let $0 < \gamma \leq 1$ and 
$(v_j)_{j=1}^{\infty}$ be a sequence in $L^2(R,X_0)$ satisfying
$$\sup_{j} \bigg( \int_{-\infty}^{+\infty} \|v_j\|_{X_0}^2 dt \bigg) < \infty, ~~~~ 
\sup_{j}\bigg( \int_{-\infty}^{+\infty} |\xi|^{2\gamma}\|\hat{v}_j\|_{X}^2 d\xi \bigg) < \infty ,$$ 
where 
$$\hat{v_j}(\xi) = \int_{-\infty}^{\infty} v_j(t)\exp (-2\pi i \xi t)dt$$
is the Fourier transformation of $v_j(t)$ on the time variable. Then there exists a subsequence of $(v_j)_{j=1}^{\infty}$
which converges strongly in $L^2(R;X)$ to some $v \in L^2(R;X)$. 
\end{lemma}

Our main result of this section is given as follows

\begin{theorem} \label{thm1}
Suppose that $\beta \geq 1$, $u_0 \in H$ and $f \in L^2(0,T;H)$. Then for any given $T > 0$, there exists 
a function $u(x,t)$  such that
\begin{equation}
u \in L^2(0,T;V) \cap L^{\infty}(0,T;H) \cap L^{\beta + 1}(0,T;L^{\beta + 1}(\Omega)).
\end{equation}
and $u(x,t)$ will satisfy 
\begin{align} 
   (u',w) &+ 2\mu ((u,w)) + b(u,u,w) + (\vartheta |u|^{\beta -1}u,w)  
          + \mu \int_{\partial \Omega} \alpha (u \cdot \tau) (w \cdot \tau) dS \nonumber  \\&= (f,w)~~~~~ \forall w \in V.
   \end{align}
\end{theorem}

\begin{proof}
We use the Galerkin approximations to prove the theorem. 
It follows from Lemma \ref{basis} that there exists an orthonormal basis for  $\mathcal{W}$. 
 We denote the basis as  $\{ w_i \in  H^2(\Omega)^3 \}$. We note that $\{ w_i\}$ also forms a  basis for $V$.
 For each positive integer $m$, we define an approximate solution $u_m$ to (\ref{wsol}) as follows:
\begin{align}
u_m(t) = \sum_{i=1}^m g_{im}(t)w_i \label{1}
\end{align}
Substituting $u_m$ in (\ref{wsol}), we obtain
\begin{align}
(u_m'(t),w_j) &+ 2\mu ((u_m(t),w_j)) + b(u_m(t),u_m(t),w_j)+(\vartheta |u_m|^{\beta - 1}u_m(t), w_j) \nonumber\\
              &+ \mu \int_{\partial \Omega} \alpha (u_m(t) \cdot \tau) (w_j \cdot \tau) dS =(f(t),w_j) \quad j= 1,...,m, \label{2}\\
u_m(0)& = \sum_{j=1}^m (u_0,w_j)w_j = u_{0m}. \label{3}
\end{align}
Further the equations (\ref{2}) --(\ref{3}) can be rewritten as  
\begin{align}
&\sum_{i=1}^m (w_i,w_j)g_{im}'(t)+ 2 \mu \sum_{i=1}^m ((w_i,w_j))g_{im}(t) + \sum_{i,l=1}^m b(w_i,w_l,w_j)g_{im}(t)g_{lm}(t) \nonumber\\
&+ \sum_{i=1}^m \vartheta |u_m|^{\beta -1}(w_i,w_j)g_{im}(t)+ \mu \int_{\partial \Omega} \alpha (u_m(t) \cdot \tau)(w_j \cdot \tau)dS = (f(t),w_j),\nonumber \\
& j= 1,...,m. \label{4} \\
&u_m(0) = \sum_{j=1}^m (u_0,w_j)w_j = u_{0m}. \label{3'}
\end{align}
The equations (\ref{4})--(\ref{3'}) form a nonlinear system of differential equations in the  function $g_{1m}, ... , g_{mm}$. 
Now we apply Picard's theorem, which ensure us an existence of a unique local solution of (\ref{4})
in some interval $[0,T_m] \subset [0,T]$.

We have a priori estimates on the approximate solutions $u_m$ as follows which ensure us the existence of solution for all $T$.

\begin{lemma}\label{pl1}
Suppose that $u_0 \in H$. Then for any given $T>0$ and any $\beta \geq 1$, we have
\begin{align}
\sup_{0 \leq t \leq T} \|u_m(t)\|_{L^2}^2 + 2\mu \int_0^T \|u_m\|_{H^1}^2 dt + 2\vartheta \int_0^T \|u_m\|_{L^{\beta+1}}^{\beta + 1} dt & \leq \|u_0\|_{L^2}^2 \nonumber  + \frac{2}{\mu} \int_0^T \|f\|_{L^2}^2dt.
\end{align}
\end{lemma}

\begin{proof}
We multiply (\ref{2}) by $g_{jm}(t)$ and add these equation for $j=1,...,m$. We get
\begin{align}
(u_m'(t),& u_m(t))+2\mu \|u_m(t)\|_{H^1}^2 +\vartheta \|u_m\|_{L^{\beta+1}}^{\beta+1}+ 
\mu \int_{\partial \Omega} \alpha (u_m(t) \cdot \tau)(u_m(t) \cdot \tau)dS \nonumber \\
& = (f(t),u_m(t)),\label{5}
\end{align}
where we have used the fact that $b(u,v,v)=0 ~ , ~ \forall v \in H_0^1(\Omega)~,~ u \in V .$
\begin{align}
\frac{1}{2} &\frac{d}{dt} \|u_m(t)\|_{L^2}^2 + 2\mu \|u_m(t)\|_{H^1}^2  +\vartheta \|u_m\|_{L^{\beta+1}}^{\beta+1}+ \mu \int_{\partial \Omega} \alpha (u_m(t) \cdot \tau)^2 ds  \nonumber \\
&= (f(t),u_m(t)). 
\label{6}
\end{align}
The right-hand side of (\ref{6}) is estimated as by 
\begin{align}
 | (f(t),u_m(t))| &\leq \|f(t)\|_{V'} \|u_m(t)\|_{H^1}\nonumber \\
 &\leq \mu \|u_m(t)\|_{H^1}^2 + \frac{1}{\mu} \|f(t)\|_{V'}^2.\label{e1}
\end{align}
Using (\ref{e1}) in  (\ref{6}), we obtain 
\begin{align}
\frac{1}{2} \frac{d}{dt} \|u_m(t)\|_{L^2}^2 + \mu \|u_m(t)\|_{H^1}^2 +\vartheta \|u_m\|_{L^{\beta+1}}^{\beta+1} \leq \frac{2}{\mu} \|f(t)\|_{V'}^2. \label{7}
\end{align}
We integrate (\ref{7}) from $0$ to $T$ and obtain
\begin{align}
\sup_{0 \leq t \leq T} \|u_m(t)\|_{L^2}^2 + 2\mu \int_0^T \|u_m\|_{H^1}^2 dt + 2\vartheta \int_0^T \|u_m\|_{L^{\beta+1}}^{\beta + 1} dt & \leq \|u_0\|_{L^2}^2 \nonumber  + \frac{2}{\mu} \int_0^T \|f\|_{L^2}^2dt.
\end{align}
The proof of Lemma \ref{pl1} is finished.
\end{proof}
 
 Applying Lemma \ref{pl1}, we obtain the global existence of the approximate solutions $u_m \in 
  L^2(0,T;V) \cap L^{\infty}(0,T;H) \cap L^{\beta + 1}(0,T;L^{\beta + 1}(\Omega)) $. Next, we will use Lemma \ref{lem1}
to prove the strong convergence of $u_m$ or its subsequence in $L^2 \cap L^{\beta}([0,T] \times \Omega)$.
 Let $\tilde{u}_m$ denote the function from $\mathbb{R}$ to $V$, define as follows
 
 \[ \tilde{u}_m =
  \begin{cases}
      u_m     & \quad \text{on } [0,T]\\
      0  & \quad \text{otherwise.} 
  \end{cases}
\] 
  Similarly, we define $g_{im}(t)$ to $\mathbb{R}$ by defining 
\[ \tilde{g}_{im}(t) =
  \begin{cases}
      g_{im}    & \quad \text{on } [0,T]\\
      0  & \quad \text{otherwise.} 
  \end{cases}
\]  
 The Fourier transformations on time variable of $\tilde{u}_m$ and $\tilde{g}_{im}$ are denoted by
$\hat{\tilde{u}}_m$ and $\hat{\tilde{g}}_{im}$ respectively. We want to show that
\begin{equation}
\int_{-\infty}^{+\infty} |\xi|^{2\gamma}\|\hat{\tilde{u}}_m(\xi)\|_{L^2}^2 d\xi \leq C ,~~ \text{for some } \gamma > 0 \label{8}
\end{equation}
where $C$ is some positive constant.
Note that approximate solutions $\tilde{u}_m$ satisfy
\begin{align}
\frac{d}{dt}(\hat{u}_m,w_j) = (\hat{f}_m,w_j) &+ (\vartheta |\hat{u}_m|^{\beta -1}\hat{u}_m(t),w_j)+
(u_{0m},w_j)\delta_0 \nonumber\\
                                              &-(u_m(T),w_j)\delta_T, ~~~ j=1,...,m \label{9}
\end{align}
where $\delta_0,\delta_T$ are Dirac distributions at $0$ and $T$ and $f_m = f - \mu Au_m - Bu_m,$ $\hat{f}_m = f ~~\text{on} ~~ [0,T], ~~ 0 ~~\text{outside this interval}$. we already define $A$ and $B$ in (\ref{p12}) and (\ref{p11}) 
respectively.

Taking the Fourier transformation about the time variable, (\ref{9}) gives
\begin{align}
2\pi i \xi(\hat{\tilde{u}}_m,w_j) = (\hat{\tilde{f}}_m,w_j) &+ \vartheta (|\tilde{u}_m\hat{|^{\beta - 1} 
\tilde{u}}_{m}(\xi) ,w_j)+(u_{0m},w_j)\nonumber\\ 
&- (u_m(T),w_j)\exp (-2\pi i T \xi), \label{10}
\end{align} 
where $\hat{\tilde{f}}_m$ denotes the Fourier transformation of $\tilde{f}_m$.

We multiply (\ref{10}) by $\hat{\tilde{g}}_{jm}(\xi)$ and add the resulting equations for $j=1,...,m$, we get
\begin{align}
2\pi i \xi \|\hat{\tilde{u}}_m(\xi)\|_{L^2}^2 = (\hat{\tilde{f}}_m,\hat{\tilde{u}}_m) &+  \vartheta (|\tilde{u}_m\hat{|^{\beta - 1} \tilde{u}}_{m}(\xi) ,\hat{\tilde{u}}_m) + (u_{0m},\hat{\tilde{u}}_m) \nonumber\\
&+ (u_m(T),\hat{\tilde{u}}_m)\exp (-2\pi i T \xi). \label{11}
\end{align} 
Because of (\ref{c2}) and (\ref{c1}), for any $v \in L^2(0,T;V) \cap L^{\beta + 1}(0,T;L^{\beta + 1})$, we have
\begin{align*}
(f_m(t),v) &= (f,v)-\mu(Au_m,v)-(Bu_m,v) \nonumber \\ 
           &\leq C (\|f(t)\|_{V'}+\mu \|u_m(t)\|_{H^1}+c_1\|u_m(t)\|_{H^1}^2)\|v\|_{H^1}.
\end{align*}
It follows that for any given $T>0$
\begin{align*}
\int_0^T \|f_m(t)\|_{V'} dt \leq \int_0^T C (\|f(t)\|_{V'}+\mu \|u_m(t)\|_{H^1}+c_1\|u_m(t)\|_{H^1}^2) \leq C,
\end{align*}
hence
\begin{equation}
\sup_{\xi \in \mathbb{R}} \|\hat{\tilde{f}}_m(\xi)\|_{V'} \leq \int_0^T \|f_m(t)\|_{V'} dt \leq C. \label{12}
\end{equation}
Also, it follows from Lemma \ref{pl1} that
\begin{align*}
\int_0^T \||u_m|^{\beta - 1}u_m\|_{\frac{\beta + 1}{\beta}} dt \leq \int_0^T \|u_m\|_{\beta + 1}^{\beta} dt \leq C,
\end{align*}
which implies that
\begin{equation}
\sup_{\xi \in \mathbb{R}} \||u_m\hat{|^{\beta -1}}u(\xi)\|_{\frac{\beta + 1}{\beta}} \leq C. \label{13}
\end{equation}
From Lemma \ref{pl1}, we get
\begin{equation}
\|u_m(0)\|_{L^2} \leq C, ~~~~~~ \|u_m(T)\|_{L^2} \leq C. \label{14}
\end{equation}
We deduce from (\ref{11})--(\ref{14}) that
\begin{align*}
|\xi|\|\hat{\tilde{u}}_m(\xi)\|_{L^2}^2 \leq C \big( \|\hat{\tilde{u}}_m(\xi)\|_{H^1} + \|\hat{\tilde{u}}_m(\xi)\|_{\beta +1} \big).
\end{align*}
For any $\gamma$ fixed, $0 < \gamma < \frac{1}{4}$, we observe that
\begin{equation*}
|\xi|^{2\gamma} \leq C \frac{1+|\xi|}{1+|\xi|^{1-2\gamma}}, ~~~~~ \forall \xi \in \mathbb{R}.
\end{equation*}
Thus
\begin{align}
\int_{-\infty}^{+\infty} |\xi|^{2\gamma} \|\hat{\tilde{u}}_m(\xi)\|_{L^2}^2 d\xi 
&\leq C\int_{-\infty}^{+\infty} \frac{1+|\xi|}{1+|\xi|^{1-2\gamma}} \|\hat{\tilde{u}}_m(\xi)\|_{L^2}^2 d\xi \nonumber \\
& \leq C\int_{-\infty}^{+\infty} \|\hat{\tilde{u}}_m(\xi)\|_{L^2}^2 d\xi  + C \int_{-\infty}^{+\infty}
  \frac{\|\hat{\tilde{u}}_m(\xi)\|_{H^1}}{1+|\xi|^{1-2\gamma}} d\xi \nonumber \\
&   + C\int_{-\infty}^{+\infty}
  \frac{\|\hat{\tilde{u}}_m(\xi)\|_{\beta +1}}{1+|\xi|^{1-2\gamma}} d\xi. \label{15}
\end{align}
Using the Parseval equality and Lemma \ref{pl1}, the first integral on the right hand side of (\ref{15}) is bounded
as $m \rightarrow \infty$. By the Schwartz inequality, the Parseval equality and Lemma \ref{pl1}, we get
\begin{align}
\int_{-\infty}^{+\infty} \frac{\|\hat{\tilde{u}}_m(\xi)\|_{H^1}}{1+|\xi|^{1-2\gamma}} d\xi
\leq \Bigg( \int_{-\infty}^{+\infty} \frac{d\xi}{(1+|\xi|^{1-2\gamma})^2}  \Bigg)^{\frac{1}{2}} 
     \Bigg( \int_0^T \|u_m(\xi)\|_{H^1}^2 \Bigg)^{\frac{1}{2}} \leq C \label{16}
\end{align}
for $0 < \gamma < \frac{1}{4}$.

Similarly, when $0 < \gamma < \frac{1}{2(\beta + 1)}$, we get
\begin{align}
\int_{-\infty}^{+\infty} \frac{\|\hat{\tilde{u}}_m(\xi)\|_{\beta +1}}{1+|\xi|^{1-2\gamma}} d\xi
& \leq \Bigg( \int_{-\infty}^{+\infty} \frac{d\xi}{(1+|\xi|^{1-2\gamma})^{\frac{\beta+1}{\beta}}}  \Bigg)^{\frac{\beta}{\beta + 1}} \Bigg( \int_{-\infty}^{+\infty} \|\hat{\tilde{u}}_m(\xi)\|_{\beta +1}^{\beta + 1} d\xi  \Bigg)^{\frac{1}{\beta+1}} \nonumber \\
& \leq C\Bigg( \int_{-\infty}^{+\infty} \|\hat{\tilde{u}}_m(\xi)\|_{\beta +1}^{\frac{\beta + 1}{\beta}} d\xi  \Bigg)^{\frac{\beta}{\beta+1}} \nonumber \\
& \leq CT^{\frac{\beta-1}{\beta+1}} \Bigg( \int_0^T \|u_m(\xi)\|_{\beta+1}^{\beta+1} d\xi   \Bigg)^{\frac{1}{\beta}} \label{17}
\end{align} 
It follows from (\ref{15}) that
\begin{equation}
\int_{-\infty}^{+\infty} |\xi|^{2\gamma} \|\hat{\tilde{u}}_m(\xi)\|_{L^2}^2 d\xi \leq C. \label{18}
\end{equation}
From lemma \ref{pl1}, we can say that, there exists a function $u(x,t)$ such that
\begin{equation}
u \in L^2(0,T;V) \cap L^{\infty}(0,T;H) \cap L^{\beta + 1}(0,T;L^{\beta + 1}(\Omega)) \label{19}
\end{equation}
and there exists a subsequence of $\{u_m\}_{m=1}^{\infty}$, still denoted by $u_m$, such that 
\begin{center}
$u_m \rightarrow u$ weakly in $L^2(0,T;V)$, \\[0.2cm]
$u_m \rightarrow u$ weak-star topology of  $L^{\infty}(0,T;H)$,\\[0.2cm]
and $u_m \rightarrow u$ weakly in $L^{\beta + 1}(0,T;L^{\beta + 1}(\Omega))$. 
\end{center}
Moreover, we choose $\Omega_1 \subset \Omega_2 \subset \Omega_3 \subset \cdots$ with smooth boundary, satisfying
$\cup_{i=1}^{\infty} \Omega_i = \Omega$. For any fixed $i=1,2,...,$ we take $X_0=V,~~X=L^2(\Omega_i)$ in Lemma \ref{lem1}. 
From Lemma \ref{lem1}, Lemma \ref{pl1} and (\ref{18}), we obtain that there exists a subsequence 
of $\{u_m\}_{m=1}^{\infty}$, still denoted by itself, such that $u_m \rightarrow u$ strongly in $L^2(0,T;L^2(\Omega_i))$. 
Also, $\int_0^T \int_{\Omega} |u_m|^{\beta + 1} dx~dt \leq C$, we obtain that
$u_{m_j} \rightarrow u$ strongly in $L^p(0,T;L_{\text{loc}}^p(\Omega))$ for $2 \leq p < \beta + 1$ if $\beta > 1$.

Let $\phi$ be a continuously differentiable function on $[0,T]$ with $\phi(t)=0$. We multiply (\ref{2}) by $\phi(t)$, and 
then integrate by parts. We obtain
\begin{align}
-&\int_0^T (u_m(t),\phi'(t)w_j)dt + \int_0^T \phi(t) \bigg( 2\mu ((u_m(t),w_j)) + b(u_m(t),u_m(t),w_j) \nonumber\\ 
&+ (\vartheta |u_m|^{\beta - 1}u_m(t), w_j)+ \mu \int_{\partial \Omega} \alpha (u_m(t) \cdot \tau)(w_j \cdot \tau) dS
- (f(t),w_j) \bigg)dt\nonumber\\  
&= (u_{0m},w_j)\phi(0) \hspace{1cm} j=1,...m. \label{20}
\end{align}
We already have $u_m \rightarrow u$ weakly in $L^2(0,T;H)$. Let $v_m = Tu_m$ and $v=Tu$, where $T:W^{1,2}
(\Omega) \rightarrow L^2(\partial\Omega)$ trace operator. Since $T$ is linear and continuous $v_m \rightarrow v$
weakly in $L^2(0,T;H)$. Also we know that $W^{1-1/p,p}(\Omega)$ compactly embedded in $L^p(\partial \Omega)$. 
Then the weak convergence of $v_m$ imply strong convergence of $v_m.$
By compactness result \cite[Theorem III.2.2]{tem79} and \cite[Theorem II.6.2]{nec12}, we have $(u_m(t) \cdot \tau)
 \rightarrow (u \cdot \tau)$  strongly in $L^2( \partial \Omega)$. Passing limit to the (\ref{20}), we obtain the 
 equation
\begin{align}
-&\int_0^T (u(t),\phi'(t)v)dt + \int_0^T \phi(t) \bigg( 2\mu ((u(t),v)) + b(u(t),u(t),v) \nonumber\\ 
&+ (\vartheta |u|^{\beta - 1}u(t), v)+ \mu \int_{\partial \Omega} \alpha (u(t) \cdot \tau)(v \cdot \tau) dS
- (f(t),v) \bigg)\nonumber\\  
&= (u_{0},v)\phi(0), \label{21}
\end{align} 
holds for $v=w_1,w_2,...$ by linearity this equation holds for $v=~~$any finite linear combination of the $w_j$, and by
continuity argument (\ref{21}) is true for any $v \in V$. So, $u(x,t)$ is solution of damped Navier-Stokes system. The proof of Theorem (\ref{thm1}) is finished.  
\end{proof}

\section{Regularity of solution} \label{exstrong}
In this section we discuss about the regularity of solution obtained in section \ref{weak} of damped Navier-Stokes system (\ref{be})--(\ref{nosl}).

\begin{theorem} \label{thm2}
Let $u_0 \in \mathcal{W}$, $f \in L^{\infty}(0,T;H)$, $f' \in L^1(0,T;H)$ and if $\beta \geq 3$ and $\mu$ is large enough
or if $f$ and $u_0$ 
are small enough, then the solution $u$ from Theorem \ref{thm1} will satisfy
\begin{equation}
\partial_t u \in L^2(0,T;V) \cap L^{\infty}(0,T;H). \label{22}
\end{equation}
\end{theorem}

\begin{proof}
 We  show that the
approximate solution define in (\ref{1}) also satisfies a priori estimate:
\begin{equation}
u_m' \in L^2(0,T;V) \cap L^{\infty}(0,T;H). \label{23}
\end{equation}
In the limit (\ref{23}) implies (\ref{22}).

Since $u_0 \in \mathcal{W}$, we choose $u_{0m}$ as the orthogonal projection in $\mathcal{W}$ of $u_0$ onto the space 
spanned by $w_1,...,w_m$, then
\begin{equation}
u_{0m} \to u_0 ~~~\text{in}~~~ H^2(\Omega),~~~~\text{and}~~~~~ \|u_{0m}\|_{H^2} \leq \|u_0\|_{H^2}. \label{24}  
\end{equation} 

\begin{lemma}\label{t2l1}
Let the hypotheses of the Theorem \ref{thm2} hold and $u_m$ be the approximate solution defined 
by (\ref{1}). Then $\{u_m'(0)\}$ belongs to bounded subset of $H$.
\end{lemma}

\begin{proof}
we multiply (\ref{2}) by $g_{jm}'(t)$ and  then add resulting equation for $j=1,...,m$. We obtain
\begin{align}
\|u_m'(t)\|_{L^2}^2 &+ 2 \mu ((u_m(t),u_m'(t))) + b(u_m(t),u_m(t),u_m'(t))+(\vartheta |u_m|^{\beta - 1}u_m(t), u_m'(t))\nonumber \\
                    &+\mu \int_{\partial \Omega} \alpha (u_m(t) \cdot \tau)(u_m'(t) \cdot \tau) dS = (f(t),u_m'(t)). \label{25}
\end{align}
Considering time $t=0$, we get
\begin{align}
\|u_m'(0)\|_{L^2}^2 &+ 2 \mu ((u_m(0),u_m'(0))) + b(u_m(0),u_m(0),u_m'(0))+(\vartheta |u_m(0)|^{\beta - 1}u_m(0), u_m'(0)) \nonumber\\
                    &+ \mu \int_{\partial \Omega} \alpha (u_m(0) \cdot \tau)(u_m'(0) \cdot \tau) dS = (f(0),u_m'(0)). \label{26}
\end{align}
We note that  
\begin{align}
 &2 \mu ((u_m(0),u_m'(0))) + \mu \int_{\partial \Omega} \alpha (u_m(0) \cdot \tau) 
(u_m'(0) \cdot \tau) dS \nonumber\\
& = -\mu \int_{\Omega} \Delta u_m(0) u_m'(0)dx \nonumber\\
 & = (-\mu \Delta u_m(0), u_m'(0)).\label{27}
\end{align}
From the definition of $w_j$ in $u_m$, we get the following estimate
\begin{align}
 \|\Delta u_m(0)\|_{L^2}& \leq c_1 \|u_m(0)\|_{H^2} \nonumber \\&\leq c_1 \|u(0)\|_{H^2}\label{28}
\end{align}
 for some positive constant $c_1.$
Moreover,for some positive constant $c_2$ we get from Lemma \ref{pre} is that
\begin{align}
|b(u_m(0),u_m(0),u_m'(0))| \leq c_2 \|u_m(0)\|_{H^2}^2\|u_m(0)\|_{L^2}.
\end{align} 
Using (\ref{27}) and (\ref{28}) in (\ref{26}), we obtain
\begin{align*}
\|u_m'(0)\|_{L^2}^2 &= (f(0),u_m'(0)) + \mu (\Delta u_m(0), u_m'(0)) - b(u_m(0),u_m(0),u_m'(0)) \\
                    & \hspace{1.5cm}-(\vartheta |u_m|^{\beta - 1}u_m(0), u_m'(0)) \\
& \leq \bigg (\|f(0)\|_{L^2} + \mu c_1 \|u_m(0)\|_{H^2} + c_2\|u_m(0)\|_{L^2}^2 \\
& \hspace{1.5cm}+\vartheta |u_m|^{\beta - 1}\|u_m(0)\|_{L^2} \bigg ) \|u_m'(0)\|_{L^2}. 
\end{align*}
It follows that 
\begin{align}
 \|u_m'(0)\|_{L^2} &\leq  \|f(0)\|_{L^2} + \mu c_1 \|u_m(0)\|_{H^2} + c_2\|u_m(0)\|_{L^2}^2 +\vartheta |u_m|^{\beta - 1}\|u_m(0)\|_{L^2} \nonumber\\
 &=d_1,\label{29}
\end{align}
where $d_1$ is a finite positive constant. Hence $\{u_m'(0)\}$ belongs to bounded subset of $ H.$
The proof of Lemma \ref{t2l1} is finished.
\end{proof}

\begin{lemma} \label{t2l2}
Let the hypotheses of the Theorem \ref{thm2} hold and $u_m$ is the approximate solution defined by (\ref{1}). 
Then $\{u_m'\}$ belongs to a bounded subset of $L^{\infty}(0,T;H)$ and $L^2(0,T;V)$.
\end{lemma}

\begin{proof}
As $f' \in L^1(0,T;H)$, we differentiate (\ref{2}) in the $t$ variable, we obtain
\begin{align}
& (u_m''(t),w_j) + 2\mu ((u_m'(t),w_j)) + b(u_m'(t),u_m(t),w_j) + + b(u_m(t),u_m'(t), w_j) \nonumber \\
&  + (\vartheta |u_m|^{\beta - 1}u_m'(t),w_j)+ (\vartheta (\beta - 1) |u_m|^{\beta - 3}(u_m \cdot u_m')u_m(t),w_j)\nonumber \\
&+ \mu \int_{\partial \Omega} \alpha (u_m'(t) \cdot \tau)(w_j \cdot \tau) dS = (f'(t),w_j)  \quad j= 1,..., m. \label{30}
\end{align}
We multiply (\ref{30}) by $g_{jm}'(t)$ and adding these equations for $j=1,...,m$, we obtain
\begin{align}
&\frac{1}{2}\frac{d}{dt}\|u_m'(t)\|_{L^2}^2+2\mu \|u_m'(t)\|_{H^1}^2 + b(u_m'(t),u_m(t),u_m'(t))+ (\vartheta |u_m|^{\beta - 1}u_m'(t),u_m'(t))\nonumber  \\
&\hspace{.5cm}+(\vartheta (\beta - 1) |u_m|^{\beta - 3}(u_m \cdot u_m')u_m(t),u_m'(t))+
\mu \int_{\partial \Omega} \alpha (u_m'(t) \cdot \tau)^2dS \nonumber  \\
& \hspace{.5cm}= (f'(t),u_m'(t)). \label{31}
\end{align}
(\ref{31}) can be written as follow
\begin{align}
&\frac{1}{2}\frac{d}{dt}\|u_m'(t)\|_{L^2}^2+2\mu \|u_m'(t)\|_{H^1}^2 + b(u_m'(t),u_m(t),u_m'(t))+ \vartheta |u_m|^{\beta - 1}\|u_m'(t)\|_{L^2}^2\nonumber  \\
&\hspace{.5cm}+\vartheta (\beta - 1) (|u_m|^{\beta - 3}(u_m \cdot u_m')u_m(t),u_m'(t))+
\mu \int_{\partial \Omega} \alpha (u_m'(t) \cdot \tau)^2dS \nonumber  \\
& \hspace{.5cm}= (f'(t),u_m'(t)). \label{32}
\end{align}
Note that $\vartheta (\beta - 1) (|u_m|^{\beta - 3}(u_m \cdot u_m')u_m(t),u_m'(t)) \geq 0$ when $\beta \geq 3$. Since
\begin{align*}
& \vartheta (\beta - 1) (|u_m|^{\beta - 3}(u_m \cdot u_m')u_m(t),u_m'(t)) \\
& = \vartheta (\beta - 1) \int_{\Omega} (|u_m|^{\beta - 3}(u_m \cdot u_m') u_m(t) \cdot u_m'(t)) dx \\
& = \vartheta (\beta - 1)  \int_{\Omega} |u_m|^{\beta - 3} (u_m(t) \cdot u_m'(t))^2 dx
\end{align*} 
Now (\ref{32}) becomes
\begin{align}
&\frac{1}{2}\frac{d}{dt}\|u_m'(t)\|_{L^2}^2+2\mu \|u_m'(t)\|_{H^1}^2 + b(u_m'(t),u_m(t),u_m'(t)) 
 \leq (f'(t),u_m'(t)) \label{33}
\end{align}
It follows from (\ref{p2}) that 
\begin{align*}
|b(u_m'(t),u_m(t),u_m'(t))|  
&\leq c_3 \|u_m'(t)\|_{H^1}\hspace{.1cm} \|u_m(t)\|_{H^1}\hspace{.1cm}\|u_m'(t)\|_{H^1}\\
&= c_3\|u_m'(t)\|_{H^1}^2\|u_m(t)\|_{H^1}. 
\end{align*}
Using the above estimates in (\ref{33}), we get
\begin{align}
\frac{d}{dt}\|u_m'(t)\|_{L^2}^2+2(2\mu -c_3\|u_m(t)\|_{H^1})\|u_m'(t)\|_{H^1}^2 \leq 2(f'(t),u_m'(t)). \label{34}
\end{align}
From (\ref{5}) we get
\begin{align}
\mu \|u_m(t)\|_{H^1}^2  & \leq \frac{1}{\mu} \|f(t)\|_{V'}^2 - (u_m(t),u_m'(t)) - \vartheta \|u_m\|_{L^{\beta+1}}^{\beta+1} \nonumber \\
                  & \leq \frac{1}{\mu} \|f(t)\|_{V'}^2 + 2\|u_m(t)\|_{L^2}\hspace{.1cm}\|u_m'(t)\|_{L^2} +\vartheta \|u_m\|_{L^{\beta+1}}^{\beta+1} \nonumber \\
                  & \leq \frac{d_2}{\mu} + 2\Bigg (\|u_0\|_{L^2}^2+\frac{Td_2}{\mu}\Bigg)^{\frac{1}{2}} \|u_m'(t)\|_{L^2}+\vartheta \|u_m\|_{L^{\beta+1}}^{\beta+1}, \label{35}
\end{align}
where $d_2=\|f\|_{L^{\infty}(0,T;V')}^2.$ 
So,
\begin{equation}
\mu \|u_m(0)\|_{H^1}^2 \leq \frac{d_2}{\mu} + 2 (\|u_0\|_{L^2}^2+\frac{Td_2}{\mu})^{\frac{1}{2}} d_1 +\vartheta \|u_0\|_{L^{\beta+1}}^{\beta+1} = d_3. \label{36}
\end{equation}
For $\mu$ is large enough or $f$ and $u_0$ 
are small enough, we define $d_4$ such that 
\begin{equation}
d_4 = \frac{d_2}{\mu} + (1+d_1^2)(\|u_0\|_{L^2}^2+\frac{Td_2}{\mu})^{\frac{1}{2}} \exp(\int_0^T\|f'(s)\|_{L^2}ds)+\vartheta \|u_m\|_{L^{\beta+1}}^{\beta+1}
< \frac{\mu^3}{c_3^2} \label{37}
\end{equation}
for some positive constant $c_3.$ Then  $$d_3 \leq d_4.$$
Thus from (\ref{36})-(\ref{37}), it follows that 
\begin{equation}
\mu \|u_m(0)\|_{H^1}^2 \leq d_3 \leq d_4 < \frac{\mu^3}{c_3^2}, \nonumber
\end{equation}
and
\begin{align*}
 &~~~~\|u_m(0)\|_{H^1}^2 < \frac{\mu^2}{c_3^2} \\
&\Rightarrow \mu^2 > c_3^2 \|u_m(0)\|_{H^1}^2 \\
&\Rightarrow 2\mu - c_3 \|u_m(0)\|_{H^1} >0 .
\end{align*}
That is  $2\mu - c_3 \|u_m(t)\|_{H^1}>0, ~~t \in I$, where $I$ is an  interval containing $0$. Let $T_m$ be the the first time $t$  such that
$t \leq T$ and 
$$2\mu - c_3 \|u_m(T_m)\|_{H^1}=0.$$
If this is not the case, then we must have $T_m=T$. 
\begin{equation}
2\mu - c_3 \|u_m(t)\|_{H^1} \geq 0, \quad 0 \leq t \leq T_m. \label{38} 
\end{equation} 
Now from (\ref{34}) using (\ref{38}) we obtain,
\begin{equation}
\frac{d}{dt}\|u_m'(t)\|_{L^2}^2 \leq 2 \|f'(t)\|_{L^2}\hspace{.1cm}\|u_m'(t)\|_{L^2}. \nonumber
\end{equation}Further, we have the following estimate
\begin{equation}
\frac{d}{dt}(1+\|u_m'(t)\|_{L^2}^2) \leq \|f'(t)\|_{L^2} (1+\|u_m'(t)\|_{L^2}^2).
\end{equation}
Applying Gronwall's lemma and using (\ref{29}), we obtain
\begin{align}
1+\|u_m'(t)\|_{L^2}^2 & \leq (1+\|u_m'(0)\|_{L^2}^2) \exp (\int_0^t \|f'(s)\|_{L^2}ds) \nonumber \\
                & \leq (1+d_1^2) \exp (\int_0^t \|f'(s)\|_{L^2}ds).
\end{align}
From (\ref{35}), we get
\begin{align}
\mu \|u_m(t)\|_{H^1}^2 & \leq \frac{d_2}{\mu} + 2 (\|u_0\|_{L^2}^2+\frac{Td_2}{\mu})^{\frac{1}{2}} \|u_m'(t)\|_{L^2}+
\vartheta \|u_m\|_{L^{\beta+1}}^{\beta+1} \nonumber \\
                 & \leq \frac{d_2}{\mu} + (\|u_0\|_{L^2}^2+\frac{Td_2}{\mu})^{\frac{1}{2}} (1+\|u_m'(t)\|_{L^2}^2)+
\vartheta \|u_m\|_{L^{\beta+1}}^{\beta+1} \nonumber \\
                 & \leq \frac{d_2}{\mu} + (\|u_0\|_{L^2}^2+\frac{Td_2}{\mu})^{\frac{1}{2}}
                    (1+d_1^2) \exp (\int_0^t \|f'(s)\|_{L^2}ds)+ \vartheta \|u_m\|_{L^{\beta+1}}^{\beta+1} \nonumber \\
                  &   = d_4.                  
\end{align}
Thus for $0<t \leq T_m $, we have
\begin{align}
& \mu \|u_m(t)\|_{H^1}^2 \leq d_4, \quad \nonumber \\
&\Rightarrow
\|u_m(t)\|_{H^1} \leq \sqrt{\frac{d_4}{\mu}}. \nonumber
\end{align}
Now
\begin{align}
2\mu - c_3 \|u_m(t)\|_{H^1} & \geq 2\mu - c_3 \sqrt{\frac{d_4}{\mu}}\\& > \mu , \quad 0\leq t \leq T_m.
\end{align}
Then $T_m=T$, and (\ref{34}) implies
\begin{align}
\frac{d}{dt}\|u_m'(t)\|_{L^2}^2 + 2(2\mu - c_3\|u_m(t)\|_{H^1})\|u_m'(t)\|_{H^1}^2 & \leq 2\|f'(t)\|_{L^2}\hspace{.1cm}\|u_m'(t)\|_{L^2},
\quad 0 \leq t \leq T. \label{40}
\end{align}
Using Gronwall's Lemma, we can conclude that $u_m' $ is a bounded set of $L^2(0,T;V) \cap L^{\infty}(0,T;H)$. 
The proof of Lemma \ref{t2l2} is finished.
\end{proof}
It follows from  Lemma \ref{t2l1} and Lemma \ref{t2l2}, the proof of the theorem is completed.
\end{proof}

\section{Uniqueness of solution} \label{unique}
In this section, we prove the uniqueness theorem for the solution obtained in section \ref{exstrong} of the system (\ref{be})--(\ref{nosl}). 
\begin{theorem}\label{thm3}
Let $u_0 \in \mathcal{W}$, $f \in L^{\infty}(0,T;H)$, $f' \in L^1(0,T;H)$. Then there is atmost one solution $u$
of the damped Navier-Stokes system (\ref{be})--(\ref{nosl}) such that
\begin{equation*}
u \in L^{\infty}(0,T;L^4(\Omega)), ~~~~ u' \in L^2(0,T;V) \cap L^{\infty}(0,T;H).
\end{equation*}
\end{theorem}

\begin{proof}
Let us assume that $u_1$ and $u_2$ are two solutions, and let $u=u_1 - u_2$. The difference $u=u_1 - u_2$
satisfies
\begin{align}
 u' + \mu Au + \vartheta|u_1|^{\beta-1}u_1-\vartheta|u_2|^{\beta-1}u_2 &= -Bu_1 + Bu_2,\label{41}\\
 u(0)&=0.\nonumber
\end{align} 
Taking the scalar product of (\ref{41}) with $u(t)$, we obtain
\begin{align}
\frac{1}{2} &\frac{d}{dt} \|u(t)\|_{L^2}^2 + 2\mu \|u(t)\|_{H^1}^2 + \int_{\Omega} \bigg( |u_1|^{\beta-1}u_1 - |u_2|^{\beta-1}u_2\bigg)(u_1-u_2) dx \nonumber \\
&+ \mu \int_{\partial \Omega} \alpha (u(t) \cdot \tau)^2 dS = 2 b(u_2(t),u_2(t),u(t))-2b(u_1(t),u_1(t),u(t)). \label{42}
\end{align}
Note that
\begin{align*}
\int_{\Omega}& \bigg( |u_1|^{\beta-1}u_1 - |u_2|^{\beta-1}u_2\bigg)(u_1-u_2) dx \\ 
&\geq \int_{\Omega} \bigg(  |u_1|^{\beta+1} - |u_2|^{\beta}|u_1|-|u_1|^{\beta}|u_2| + |u_2|^{\beta+1} \bigg) dx \\
& = \int_{\Omega} (|u_1|^{\beta}-|u_2|^{\beta})(|u_1|-|u_2|) dx \geq 0.
\end{align*}
Also
\begin{align}
2b(u_2,u_2,u_1-u_2)-2b(u_1,u_1,u_1-u_2) & = 2b(u_2,u_2,u_1)+2b(u_1,u_1,u_2) \nonumber \\
                                        & = 2b(u_2,u_2,u_1)-2b(u_1,u_2,u_1) \nonumber \\
                                        & = -2b(u_1,u_2,u_1) + 2b(u_2,u_2,u_1) \nonumber \\
                                        & \hspace{.4cm}+ 2b(u_1,u_2,u_2) -2b(u_2,u_2,u_2) \nonumber \\
                                        & = -2b(u_1-u_2,u_2,u_1) + 2b (u_1-u_2,u_2,u_2) \nonumber \\
                                        & = -2b(u_1-u_2,u_2,u_1-u_2) \nonumber \\
                                        & = 2b(u,u,u_2).    \label{43}                                      
\end{align} 
Now (\ref{42}) becomes
\begin{align}
\frac{1}{2} &\frac{d}{dt} \|u(t)\|_{L^2}^2 + 2\mu \|u(t)\|_{H^1}^2 \leq 2b(u,u,u_2). \label{44} 
\end{align}
By the H\"{o}lder inequality and (\ref{p2})
\begin{align}
|b(u,u,v)| \leq k_1 \|u\|_{L^2}^{\frac{1}{4}} \|u\|_{H^1}^{\frac{7}{4}} \|v\|_{L^4(\Omega)} .\label{45}
\end{align} 
Now (\ref{44}) becomes
\begin{align*}
\frac{1}{2} \frac{d}{dt} \|u(t)\|_{L^2}^2 + 2\mu \|u(t)\|_{H^1}^2  
      & \leq 2 k_1 \|u\|_{L^2}^{\frac{1}{4}} \|u\|_{H^1}^{\frac{7}{4}} \|u_2(t)\|_{L^4(\Omega)} \nonumber \\
      & \leq 2 \mu \|u(t)\|_{H^1}^2  + k_2\|u(t)\|_{L^2}^2 \|u_2(t)\|_{L^4(\Omega)}^8 .
\end{align*}
So we get
\begin{align*}
 \frac{d}{dt} \|u(t)\|_{L^2}^2 \leq k \|u(t)\|_{L^2}^2 \|u_2(t)\|_{L^4(\Omega)}^8 . 
\end{align*}
Applying Gronwall's Lemma, we conclude that
\begin{align*}
\|u(t)\|_{L^2}^2 &\leq \|u(0)\|_{L^2}\exp \Bigg(\int_0^t k\|u_2(t)\|_{L^4(\Omega)}^8 dt \Bigg)\\
&=0.
\end{align*}
Thus $u_1(t)=u_2(t)$. The proof of the Theorem \ref{thm3} is finished.
\end{proof}

\textbf{Acknowledgements}\\
 The first author acknowledge the grant from NBHM (Grant No.02011/9/2019\\
 NBHM(R.P.)/R ~and ~ D II/1324) and 
 the second author is supported by the DST MATRICS 
(Grant No. SERB/F/12082/2018-2019).

{\bf Author name and affiliations:}\\
 Rajib Haloi, Department of Mathematical  Sciences,\\
  Tezpur University, Sonitpur\\
   ASSAM, INDIA, Pin- 784028.\\
E-mail:rajib.haloi@gmail.com\\
Phone number +913712-275511,\\
 Fax: +913712-267006.\\
 Subha Pal, Department of Mathematical  Sciences,\\
  Tezpur University, Sonitpur\\
   ASSAM, INDIA, Pin- 784028.\\
E-mail:sp234sp@gmail.com\\
Phone number +913712-275511,\\
 Fax: +913712-267006.

\end{document}